\documentclass[11pt, a4paper]{article}
\usepackage{stmaryrd}
\usepackage{}

\usepackage{lineno}

\usepackage{graphicx}
\usepackage{ae}
\usepackage{amsmath}
\usepackage{amssymb}
\usepackage{latexsym}
\usepackage{url}
\usepackage{epsfig}

\usepackage{cite}
\usepackage{mathrsfs}
\usepackage{amsfonts}
\usepackage{amsthm}

\def\qed{\hfill$\Box$\vspace{12pt}}

\usepackage{color}

\input amssym.def
\input amssym.tex

\long\def\delete#1{}

\newcommand{\be}{\begin{equation}}
\newcommand{\ee}{\end{equation}}
\newcommand{\bea}{\begin{eqnarray}}
\newcommand{\eea}{\end{eqnarray}}
\newcommand{\bean}{\begin{eqnarray*}}
\newcommand{\eean}{\end{eqnarray*}}

\def\deg{{\rm deg}}

\newtheorem{thm}{Theorem}[section]
\newtheorem{cor}[thm]{Corollary}
\newtheorem{lem}[thm]{Lemma}
\newtheorem{prop}[thm]{Proposition}

\newtheorem{rem}{Remark}
\numberwithin{equation}{section}

\title{Laplacian spectral characterization of some double starlike trees\thanks{Supported by the Natural Science Foundation of China (No.11361033).}}

\author{Pengli Lu$^{1}$
\,and\, Xiaogang Liu$^{2,}$\thanks{Corresponding author.   E-mail
addresses: lupengli88@163.com (\textbf{P. Lu}), xiaogliu.yzhang@gmail.com (\textbf{X. Liu}).}
\\
\footnotesize{1. School of Computer and Communication, Lanzhou
University of Technology, Lanzhou, 730050, Gansu, P.R. China}\\
\footnotesize{2. Department of Mathematics and Statistics, The University of Melbourne, Parkville, VIC 3010, Australia}}
\date{}
\begin{document}
\maketitle

\begin{abstract}
A tree is called \emph{double starlike} if it has exactly two vertices of degree greater than two. Let $H(p,n,q)$ denote the double starlike tree obtained by attaching $p$ pendant vertices to one pendant vertex of the path $P_n$ and $q$ pendant vertices to the other pendant vertex of $P_n$. In this paper, we prove that $H(p,n,q)$ is determined by its Laplacian spectrum.
\\
\\
\textbf{keywords:} Adjacency spectrum; Laplacian spectrum; $A$-cospectral graphs; $L$-cospectral graphs; Line graph
\\
\end{abstract}

\section{Introduction}

All graphs considered in this paper are simple and undirected. Let $G=(V(G),E(G))$ be a graph  with vertex set $V(G)=\{v_1,v_2,\ldots,v_n\}$ and edge set $E(G)$, where
$v_1,v_2,\ldots,v_n$ are indexed in the non-increasing order of
degrees. Let $d_i=d_i(G)=d_G(v_i)$ be the degree of the vertex $v_i$, and
$$
\deg(G) = (d_1,d_2,\ldots,d_n)
$$
the non-increasing degree sequence of $G$. The \emph{adjacency matrix} of $G$, denoted by $A(G)$, is the $n \times n$ matrix whose $(i,j)$-entry is $1$ if $v_i$ and $v_j$ are adjacent and $0$ otherwise. We call
$L(G)=D(G)-A(G)$ the \emph{Laplacian matrix} of $G$, where
$D(G)$ is the $n\times n$ diagonal matrix with $d_1,d_2,\ldots,d_n$ as diagonal entries. The
eigenvalues of $A(G)$ and $L(G)$ are called the \emph{adjacency
eigenvalues} and \emph{Laplacian eigenvalues} of $G$, respectively.
Denote by $\lambda_1(G)\geq\lambda_2(G)\geq\cdots\geq\lambda_n(G)$ and
$\mu_1(G)\geq\mu_2(G)\geq\cdots\geq\mu_n(G)$ the adjacency
eigenvalues and the Laplacian eigenvalues of $G$, respectively. The
multiset of eigenvalues of $A(G)$ (respectively, $L(G)$) is called the
\emph{adjacency} (respectively, \emph{Laplacian}) \emph{spectrum} of $G$. Two
graphs are said to be \emph{$A$-cospectral} (respectively,  \emph{$L$-cospectral}) if they have the same adjacency (respectively, Laplacian) spectrum. A graph is said to be \emph{determined by its Laplacian} (respectively, \emph{adjacency}) \emph{spectrum} if there is no other non-isomorphic graphs $L$-cospectral (respectively, $A$-cospectral) with it.

Which graphs are determined by their spectra? This is a classical question in spectral graph theory, which was raised by G\"{u}nthard and Primas \cite{kn:Gunthard56} in 1956 with motivations from chemistry. For its background, please refer to \cite{kn:vanDam03,kn:vandam07,kn:Cvetkovic2010}. It is well known that this question is still far from being completely solved, since it is often very challenging to check whether an arbitrary given graph is determined by its spectrum or not, even for some simple-looking graphs.  Up until now, many graphs have been proved to be determined by their (adjacency or/and Laplacian) spectra \cite{kn:Aalipour13,kn:Boulet08,kn:Boulet088,kn:Boulet09,kn:Bu-Zhou-Li12,kn:Cvetkovic80,kn:Cvetkovic2010,kn:vanDam03,kn:vandam07, kn:Haemers08,kn:LuLiu09,kn:LiuM10,kn:LiuDouble,kn:Wangliu10,kn:Mirzakhah10,kn:Omidi07,kn:Shen05,kn:Stanic09,kn:Wang06,kn:Wang10, kn:Zhou12}. However, only few trees with special structures have been proved to be determined by their Laplacian spectra. We collect these known trees determined by their Laplacian spectra in the following:
\begin{enumerate}
  \item Any \emph{path} is determined by its Laplacian spectrum \cite{kn:vanDam03}.
  \item The trees $Z_n$, $T_n$ and $W_n$ are determined by their Laplacian spectra, respectively  \cite{kn:Shen05}, where $Z_n$ denotes the tree obtained by attaching two pendant vertices to one pendant vertex of $P_n$ (the path with $n$ vertices); $T_n$ denotes the tree obtained by attaching one pendant vertex to the vertex with distance $2$ from one pendant vertex of $P_{n+1}$; and $W_n$ denotes the tree obtained by attaching two pendant vertices to each pendant vertex of $P_n$.
  \item Any \emph{$T$-shape tree} is determined by its Laplacian spectrum \cite{kn:Wang06}, where a tree which has exactly one vertex of degree equal to three is said to be \emph{$T$-shape}.
  \item Any \emph{starlike tree} is determined by its Laplacian spectrum \cite{kn:Omidi07}, where a tree is said to be \emph{starlike} if it has exactly one vertex of degree greater than two.
  \item Any \emph{centipede} is determined by its Lapalcian spectrum \cite{kn:Boulet088}, where a \emph{centipede} is a tree obtained by appending a pendant vertex to each vertex of degree $2$ of a path.
  \item Any tree $M_{a,b,c}\,(a,b,c\ge1)$ satisfying $b\notin\{1,3\}$ is determined by its Laplacian spectrum \cite{kn:Stanic09}, where $M_{a,b,c}$ is the tree obtained by attaching a pendant vertex to the vertices with distance $a+1$ and $a+b+1$ from one pendant vertex of $P_{a+b+c+1}$, respectively.
  \item Any tree $H_n(p,p)$ ($n\geq 2$, $p\geq 1$) is determined by its Laplacian spectrum \cite{kn:LiuDouble}, where $H_n(p,p)$ is the tree obtained by attaching $p$ pendant vertices to each pendant vertex of $P_n$.
  \item Any tree $T^2_n$ is determined by its Laplacian spectrum \cite{kn:Bu-Zhou-Li12}, where $T^2_n$ denotes the tree obtained by attaching two pendant vertices to every vertex of $P_n$.
  \item Any \emph{banana tree} $B_{n,k}$ satisfying $n^2\le k$ is determined by its Laplacian spectrum \cite{kn:Aalipour13}, where $B_{n,k}$ denotes the tree obtained by joining a vertex to one arbitrary pendant vertex of each copy of $n$-copies of the complete bipartite graph $K_{1,k}$.
\end{enumerate}

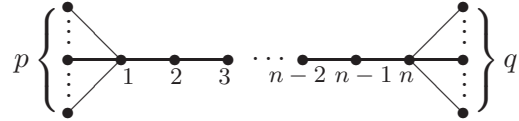
\begin{figure}[here]
\setlength{\unitlength}{4pt}
\begin{picture}(40,10)(-38.5,-4.8)
\put(0,5){\circle*{1.1}}
\put(-5,-0.7){$p\left\{\begin{array}{c}{~}\\{~}\\{~}\end{array}\right.$}
\put(34.2,-0.7){$\left.\begin{array}{c}{~}\\{~}\\{~}\end{array}\right\}q$}
\put(0,0){\circle*{1.1}}   \put(5.1,-2.3){\footnotesize{$1$}}
\put(9.6,-2.3){\footnotesize{$2$}}  \put(14.1,-2.3){\footnotesize{$3$}}
\put(17.2,0){$\ldots$}  \put(18.8,-2.1){\footnotesize{$n-2$}}
\put(25,-2.1){\footnotesize{$n-1$}} \put(31,-2.1){\footnotesize{$n$}}
\put(-0.25,1.5){$\vdots$}  \put(-0.25,-3.5){$\vdots$}
\put(36.75,1.5){$\vdots$}  \put(36.75,-3.5){$\vdots$}
\put(0,-5){\circle*{1.1}} \put(5,0){\circle*{1.1}}
\put(10,0){\circle*{1.1}} \put(15,0){\circle*{1.1}}
\put(22,0){\circle*{1.1}} \put(27,0){\circle*{1.1}}
\put(32,0){\circle*{1.1}} \put(37,0){\circle*{1.1}}
\put(37,5){\circle*{1.1}} \put(37,-5){\circle*{1.1}}
\put(0,5){\line(1,-1){5}} \put(0,0){\line(1,0){15}}
\put(22,0){\line(1,0){15}} \put(0,-5){\line(1,1){5}}
\put(32,0){\line(1,1){5}} \put(32,0){\line(1,-1){5}}
\end{picture}
\caption{The double starlike tree $H(p,n,q)$} \label{f1}
\end{figure}

We call a tree \emph{double starlike} if it has exactly two vertices of degree greater than two. Denote by $H(p,n,q)$ the double starlike tree obtained by attaching $p$ pendant vertices to one pendant vertex of $P_n$ and $q$ pendant vertices to the other pendant vertex of $P_n$ (shown in Fig. \ref{f1}). Without loss of generality, we suppose that $p\ge q\ge1$. In this paper, we prove that any double starlike tree $H(p,n,q)$ is determined by its Laplacian spectrum. We state our main result as follows.

\begin{thm}\label{mainthm}
Every double starlike tree $H(p,n,q)$ is determined by its Laplacian spectrum.
\end{thm}

\section{Preliminaries}
In this section we collect some known results that will be used in the proof of Theorem \ref{mainthm}.

\begin{lem}\label{spectrum}\emph{\cite{kn:vanDam03,kn:Oliveira02}}
For any graph, the following can be determined by its adjacency (or Laplacian) spectrum.
\begin{itemize}
\item[\rm (a)] The number of vertices.
\item[\rm (b)] The number of edges.
\end{itemize}
For any graph, the following can be determined by its adjacency spectrum.
\begin{itemize}
\item[\rm (c)] The number of closed walks of any length.
\end{itemize}
For any graph, the following can be determined by its Laplacian spectrum.
\begin{itemize}
\item[\rm (d)] The number of components.
\item[\rm (e)] The number of spanning trees.
\item[\rm (f)] The sum of the squares of degrees of vertices.
\end{itemize}
\end{lem}

\begin{lem}\label{4 cycles}\emph{\cite{kn:Cvetkovic87}} For any graph, the number of closed walks
of length $4$ is equal to twice the number of edges plus four times the
number of induced paths of length $2$ plus eight times the number
of $4$-cycles.
\end{lem}

Let $G$ be a graph. The \emph{line graph} $\mathcal{L}(G)$ of $G$ is the graph with vertices the edges of $G$ such that two vertices are adjacent in $\mathcal{L}(G)$ if and only if the corresponding edges have a common end-vertex in $G$.

\begin{lem}\label{Line graph}\emph{\cite{kn:Gutman99}} Let $T$ be a tree with $n$
vertices and $\mathcal {L}(T)$ its line graph. Then $\mu_i(T)=\lambda_i(\mathcal {L}(T))+2$ for
$i=1,2,\ldots,n-1$.
\end{lem}

\begin{lem}\label{Deleting}\emph{\cite{kn:Lotker07}}
Let $u$ be a vertex of G and $G-u$ the subgraph obtained from
$G$ by deleting $u$ together with its incident edges. Then
\[\mu_i(G)\geq\mu_{i}(G-u)\geq\mu_{i+1}(G)-1,~~i=1,2,\ldots,n-1.\]
\end{lem}

\begin{lem}\label{InterlacingEdge}\emph{\cite{kn:Cvetkovic2010}}
Let $e$ be an edge of $G$ and $G'=G-e$ the subgraph obtained from $G$ by deleting $e$. Then
\[\mu_1(G)\geq \mu_1(G')\geq \mu_2(G)\geq \mu_2(G')\geq\cdots\geq\mu_{n-1}(G)\geq \mu_{n-1}(G')\geq\mu_n(G)=\mu_n(G')=0. \]
\end{lem}

\begin{lem}\label{interlacing}\emph{\cite{kn:vanDam03}}   Suppose that $N$ is a
symmetric $n \times n$ matrix with eigenvalues $\alpha_1\geq
\alpha_2\geq\cdots\geq\alpha_n$. Then the eigenvalues
$\alpha'_1\geq\alpha'_2\geq\cdots\geq\alpha'_m$ of a principal
submatrix of $N$ of size $m$ satisfy
$\alpha_i\geq\alpha'_i\geq\alpha_{n-m+i}$ for $i=1,2,\ldots,m$.
\end{lem}

\begin{lem}\label{largest degree}\emph{\cite{kn:Kelmans74,kn:Li98}}
Let $G$ be a graph
with $V(G)\neq\emptyset$ and $E(G)\neq\emptyset$. Then
\begin{eqnarray*}
   & &d_1+1\leq\mu_{1}(G)\leq
\max\left\{\frac{d_i(d_i+m_i)+d_j(d_j+m_j)}{d_i+d_j},~~ v_iv_j\in
E(G)\right\},
\end{eqnarray*}
where $m_i$ denotes the average of the degrees of the vertices
adjacent to $v_i$ in $G$.
 \end{lem}



\begin{lem}\label{Second Lap}\emph{\cite{kn:Li00Second}}
Let $G$ be a connected graph with $n\geq 3$ vertices. Then
$\mu_2(G)\geq d_2$.
\end{lem}

\begin{lem}\label{third Lap}\emph{\cite{kn:Guo07Third}}
Let $G$ be a connected graph with $n\geq 4$ vertices. Then
$\mu_3(G)\geq d_3-1$.
\end{lem}

\section{Proof of Theorem \ref{mainthm}}

In this section, we will prove Theorem  \ref{mainthm}. Before proceeding, we need to mention the following results.
\begin{prop}\label{partialprop}
\begin{itemize}
  \item[\rm (a)] If $n=1$, then $H(p,1,q) \cong K_{1,p+q}$, which is determined by its Laplacian spectrum \cite{kn:Omidi07}.
  \item[\rm (b)] If $n=2$ or $n=3$, then $H(p,n,q)$ is determined by its Laplacian spectrum \cite{kn:Shen06}.
  \item[\rm (c)] If $p=q=1$, then $H(1,n,1)\cong P_{n+2}$, which is determined by its Laplacian spectrum \cite{kn:van Dam03}.
  \item[\rm (d)] If $p>q=1$, then $H(p,n,1)$ is a starlike tree, which is determined by its Laplacian spectrum\cite{kn:Omidi07}.
  \item[\rm (e)] If $p=q\geq2$, then $H(p,n,p)\cong H_n(p,p)$, which is determined by its Laplacian spectrum \cite{kn:LiuDouble}.
\end{itemize}
\end{prop}
Proposition \ref{partialprop} tells us that we only need to consider whether any double starlike tree $H(p,n,q)$ for $n\geq4$ and $p> q\geq2$ is determined by its Laplacian spectrum. In the following, we first bound the largest, the second largest and the third largest Laplacian eigenvalues of such graphs.

\begin{lem}\label{first second and third-Lap}
Let $G=H(p,n,q)$ with $n\geq4$ and $p>q\geq2$. Then
\begin{itemize}
  \item[\rm (a)] $p+2 \leq \mu_1(G)\leq p+2+\dfrac{1}{p+2}$.
  \item[\rm (b)] $q+2 \leq \mu_2(G)\leq q+3+\dfrac{1}{q+2}$.
  \item[\rm (c)] $\mu_3(G)<4$.
\end{itemize}
\end{lem}
\begin{proof}
(a) The result follows from Lemma \ref{largest degree} by simple computations.
\medskip

(b) Let $u$ and $v$ be the vertices of degree $p+1$ and $q+1$ in $G$, respectively. By Lemma \ref{Deleting}, we
have
\[\mu_1(G)\geq\mu_1(G-u)\geq\mu_2(G)-1.\]
Lemma \ref{largest degree} implies that
\[\mu_1(G-u)\leq q+2+\dfrac{1}{q+2}.\]
Then
\[\mu_2(G)\leq\mu_1(G-u)+1\leq q+3+\dfrac{1}{q+2}.\]
Next, let $G_1$ be a subgraph of $G$ obtained by deleting an edge whose end-vertices are neither $u$ nor $v$. Clearly, $G_1$ has two connected components. Lemma \ref{largest degree} implies that the largest Laplacian eigenvalue of each component is at least $q+2$, that is, $\mu_2(G_1)\geq q+2$. By Lemma \ref{InterlacingEdge}, we have $\mu_2(G)\geq \mu_2(G_1) \geq q+2.$
\medskip

(c) Let $M_{uv}$ be the $(p+n+q-2)\times (p+n+q-2)$ principal submatrix of $L(G)$ formed by deleting the rows and columns corresponding to $u$ and $v$. Then the largest eigenvalue of $M_{uv}$ is less than $4$. By Lemma \ref{interlacing}, $\mu_3(G)<4$.
\qed\end{proof}

\begin{lem}\label{degreesequenceLem2}
Let $G=H(p,n,q)$ with $n\geq4$ and $p>q\geq2$. Suppose that a graph $G^{\prime}$ is $L$-cospectral with $G$. Then $G'$ is a double starlike tree with $\deg(G')=(p+1, q+1, \overbrace{2,\ldots,2}^{n-2}, \overbrace{1,\ldots,1}^{p+q})$.
\end{lem}

\begin{proof}
By (a), (b), (d) and (e) of Lemma \ref{spectrum}, $G^{\prime}$ is a tree with $n+p+q$ vertices and $n+p+q-1$ edges. Let $(d_1, d_2, \ldots, d_{n+p+q})$ be the non-increasing degree sequence of graph $G^{\prime}$, and denote by $n_i$ the number of vertices with degree $i$ in $G'$, for $i=1,2,\ldots,d_1$.  By Lemma \ref{largest degree} and (a) of Lemma \ref{first second and third-Lap}, we have $d_1+1 \leq\mu_1(G^{\prime})\leq p+2+\dfrac{1}{p+2}$. Then $d_1\le p+1$. By Lemma \ref{Second Lap} and (b) of Lemma \ref{first second and third-Lap}, we have $d_2\le \mu_2(G^{\prime})=\mu_2(G)\leq q+3+\dfrac{1}{q+2}$. Then $d_2\le q+3$. By Lemma \ref{third Lap} and (c) of Lemma \ref{first second and third-Lap}, we have $d_3\le\mu_3(G^{\prime})+1=\mu_3(G)+1<5$. Then $d_3\le 4$.

On the other hand, (a), (b) and
(f) of Lemma \ref{spectrum} imply the following equations:
\begin{eqnarray}
& &  \sum_{i=1}^{d_1}n_i=n+p+q, \label{vertex sum}\\
& &  \sum_{i=1}^{d_1}in_i=2(n+p+q-1), \label{edges sum}\\
& & \sum_{i=1}^{d_1}i^2n_i=(p+1)^2+(q+1)^2+4(n-2)+p+q.
    \label{degree sum}
\end{eqnarray}
Combining (\ref{vertex sum}), (\ref{edges sum}) and (\ref{degree sum}), we have
\begin{equation}\label{sum1}
\sum_{i=1}^{d_1}(i^2-3i+2)n_i=p^2+q^2-p-q.
\end{equation}

Note that Lemma \ref{Line graph} implies that line graphs $\mathcal {L}(G^{\prime})$ and $\mathcal {L}(G)$ are $A$-cospectral. Further, by (c) of Lemma \ref{spectrum}, $\mathcal {L}(G^{\prime})$ and $\mathcal {L}(G)$ have the same number of triangles, that is,
\begin{equation}\label{sanjiao}
\sum_{i=1}^{d_1}\begin{pmatrix}
            i \\
            3 \\
          \end{pmatrix}
 n_i=\begin{pmatrix}
            p+1 \\
            3 \\
          \end{pmatrix}
          +\begin{pmatrix}
            q+1 \\
            3 \\
          \end{pmatrix}.
\end{equation}
In the following we determine the degree sequence of $G^{\prime}$. Note that $d_1\leq p+1$, $d_2\leq q+3$ and $d_3\le4$. We consider the following cases.
\bigskip

\noindent\emph {Case 1. } $q=2$ or $q=3$. Assume that $d_1<p+1$, that is, $n_{p+1}=0$. By (\ref{sum1}) and (\ref{sanjiao}), we obtain that
\begin{eqnarray}\label{addEqu1}
 \begin{pmatrix}
            p+1 \\
            3 \\
          \end{pmatrix}
          +\begin{pmatrix}
            q+1 \\
            3 \\
          \end{pmatrix}=\sum_{i=1}^{d_1}\begin{pmatrix}
            i \\
            3 \\
          \end{pmatrix}
 n_i\leq \frac{p}{6}\sum_{i=1}^{d_1}(i-1)(i-2)n_i= \frac{p}{6}(p^2+q^2-p-q).
\end{eqnarray}
If $q=2$, that is, $p\ge3$, then by plugging $q=2$ into (\ref{addEqu1}), we get $p^2-3p+6\leq 0$, which contradicts $p\geq 3$. If $q=3$, that is, $p\ge4$, then by plugging $q=3$ into (\ref{addEqu1}) again, we have $p^2-7p+24\leq 0$, which contradicts $p\geq 4$. Thus $d_1=p+1$, that is, $n_{p+1}\geq 1$. Now, we assume that $n_{p+1}\geq 2$,  that is, there exist
at least $2$ vertices with degree $p+1$ in $G'$. Then by
(\ref{sanjiao}), we have
\begin{eqnarray*}
\begin{pmatrix}
            p+1 \\
            3 \\
          \end{pmatrix}
          +\begin{pmatrix}
            q+1 \\
            3 \\
          \end{pmatrix}=\sum_{i=1}^{d_1}\begin{pmatrix}
            i \\
            3 \\
          \end{pmatrix}n_i\geq 2\begin{pmatrix}
            p+1 \\
            3 \\
          \end{pmatrix}+\sum_{i=1}^{p}\begin{pmatrix}
            i \\
            3 \\
          \end{pmatrix}n_i,
\end{eqnarray*}
that is,
\begin{eqnarray*}
\begin{pmatrix}
            q+1 \\
            3 \\
          \end{pmatrix}
          -\begin{pmatrix}
            p+1 \\
            3 \\
          \end{pmatrix}\geq \sum_{i=1}^{p}\begin{pmatrix}
            i \\
            3 \\
          \end{pmatrix}n_i\ge0.
\end{eqnarray*}
This is a contradiction to $q<p$. Hence  $n_{p+1}=1$. Further, if $q=2$, by (\ref{sanjiao}), we have $n_i=0$ for $i=4,\ldots,p$ and $n_3=1$. By (\ref{vertex sum}) and (\ref{edges sum}), we have $n_2=n-2$ and $n_1=p+2$. So $\deg(G')=(p+1, 3, \overbrace{2,\ldots,2}^{n-2}, \overbrace{1,\ldots,1}^{p+2})$.
If $q=3$, by (\ref{sanjiao}), we have $n_4\leq1$ and $n_i=0$ for $i=5,\ldots,p$. By
(\ref{vertex sum}), (\ref{edges sum}) and (\ref{degree sum}), it is
ready to obtain that $n_4=1$, $n_3=0$, $n_2=n-2$ and
$n_1=p+3$. So $\deg(G')=(p+1, 4, \overbrace{2,\ldots,2}^{n-2}, \overbrace{1,\ldots,1}^{p+3})$.

\bigskip

\noindent\emph {Case 2. } $q\geq 4$. Clearly, $d_1\geq 4$. Otherwise, assume that $d_1=3$. Then (\ref{sum1}) and (\ref{sanjiao}) imply that $n_3=\dfrac{1}{2}(p^2+q^2-p-q)$ and $n_3=
\begin{pmatrix}
            q+1 \\
            3 \\
          \end{pmatrix}
          +\begin{pmatrix}
            p+1 \\
            3 \\
          \end{pmatrix}$,
respectively. This is a contradiction. On the other hand, by combining (\ref{sum1}) and (\ref{sanjiao}), we obtain that
\begin{eqnarray*}\label{sum2}
\sum_{i=1}^{d_1}(i-1)(i-2)(i-3)n_i=p^3+q^3-3p^2-3q^2+2p+2q.
\end{eqnarray*}
That is,
\begin{equation}\label{sum3}
6n_4+\sum_{i=5}^{d_1}(i-1)(i-2)(i-3)n_i=p(p-1)(p-2)+q(q-1)(q-2).
\end{equation}
Note that $d_3\leq4$, which implies there exist at most two vertices of degree strictly greater than $4$ in $G^{\prime}$. Consider the following cases.

\medskip

\noindent\emph {Case 2.1.} $d_1=4$ and $d_2\leq 4$. By (\ref{sum3}), we have
\begin{equation}\label{AddEquals2}
6n_4=p(p-1)(p-2)+q(q-1)(q-2).
\end{equation}
Plugging (\ref{AddEquals2}) back into (\ref{sum1}), we obtain that
\[n_3=-\frac{1}{2}p(p-1)(p-3)-\frac{1}{2}q(q-1)(q-3)<0,\]
which is a contradiction to $n_3\ge0$.

\medskip

\noindent\emph {Case 2.2.} $d_1>4$ and $d_2=4$. By (\ref{sum3}), we have
\begin{equation}\label{AddEquals3}
6n_4=p(p-1)(p-2)+q(q-1)(q-2)-(d_1-1)(d_1-2)(d_1-3).
\end{equation}
Plugging (\ref{AddEquals3}) back into (\ref{sum1}), we obtain that
\[n_3=\frac{1}{2}(d_1-1)(d_1-2)(d_1-4)-\frac{1}{2}p(p-1)(p-3)-\frac{1}{2}q(q-1)(q-3).\]
Since $d_1\leq p+1$ and  $q\geq4$, we have $n_3<0$, which again contradicts $n_3\ge0$.

\medskip

\noindent\emph {Case 2.3.} $d_1>4$ and $d_2>4$. By (\ref{sum3}), we have
\begin{equation}\label{n_4}
n_4=\begin{pmatrix}
            p \\
            3 \\
          \end{pmatrix}
 +\begin{pmatrix}
            q \\
            3 \\
          \end{pmatrix}
          -\begin{pmatrix}
            d_1-1 \\
            3 \\
          \end{pmatrix}
          -\begin{pmatrix}
            d_2-1 \\
            3 \\
          \end{pmatrix}.
\end{equation}
Plugging (\ref{n_4}) back into (\ref{sum1}), we obtain that
\begin{equation}\label{n_3}
n_3=\frac{1}{2}(d_1-1)(d_1-2)(d_1-4)-\frac{1}{2}p(p-1)(p-3)+\frac{1}{2}(d_2-1)(d_2-2)(d_2-4)-\frac{1}{2}q(q-1)(q-3).
\end{equation}
Note that $d_1\leq p+1$ and $d_2\leq q+3$. Now we assume that $d_1<p+1$, that is,
$n_{p+1}=0$. Then consider the following cases.

If $d_1=p$ and $d_2=q+3$, then by (\ref{n_4}) and (\ref{n_3}), we have
\[n_4=\frac{1}{2}p^2-\frac{3}{2}p+1-q^2,\]
and
\begin{equation}\label{AddEqua33}
n_3=-\frac{3}{2}p^2+\frac{11}{2}p-5+3q^2-2q=-\frac{1}{2}(p-2)(3p-5)+3q^2-2q.
\end{equation}
Then $n_4\geq0$ implies that $p\geq\dfrac{3+\sqrt{8q^2+1}}{2}$. Substituting $p$ into (\ref{AddEqua33}), we obtain that
\[n_3\leq-\frac{1}{2}\left(\dfrac{3+\sqrt{8q^2+1}}{2}-2\right)\left(3\times\dfrac{3+\sqrt{8q^2+1}}{2}-5\right)+3q^2-2q=-\frac{1}{2}+\frac{1}{2}\sqrt{8q^2+1}-2q<0,\]
which contradicts $n_3\ge0$.

If $d_1=p-1$ and $d_2=q+3$, then by (\ref{n_4}) and (\ref{n_3}), we
have\[n_4=p^2-4p-q^2+4=(p-2-q)(p-2+q),\] and
\[n_3=-3p^2+14p-16+3q^2-2q=-(p-2)(3p-8)+3q^2-2q.\]
Then $n_4\geq0$ implies that $p\geq q+2$. Substituting $p\geq q+2$ into the above expression of $n_3$, we have
\[n_3\leq-q(3q-2)+3q^2-2q=0.\] Therefore $p=q+2,$ and $n_3=n_4=0.$
But now $d_1=p-1=q+1<d_2=q+3,$ which is a contradiction to $d_1\ge d_2$.

If $d_1\leq p-2$ and $d_2=q+3$, then $q+3=d_2\leq d_1\leq p-2$ implies that $p\geq q+5.$ Plugging $d_1\leq p-2$ and $d_2=q+3$ back into (\ref{n_3}), we have
\[n_3\leq-\frac{9}{2}p^2+\frac{51}{2}p-37+3q^2-2q\leq-\frac{9}{2}(q+5)^2+\frac{51}{2}(q+5)-37+3q^2-2q\leq-\frac{3}{2}q^2-\frac{43}{2}q-22<0,\]
which again contradicts $n_3\ge0$.

If $d_1\leq p$ and $d_2\leq q+2$, then by (\ref{n_3}) and $p\geq q+1$, we have
\[n_3\leq -\frac{3}{2}p^2+\frac{11}{2}p-4+\frac{3}{2}q^2-\frac{5}{2}q=-\frac{1}{2}(p-1)(3p-8)+\frac{3}{2}q^2-\frac{5}{2}q\leq-\frac{1}{2}q(3q-5)+\frac{3}{2}q^2-\frac{5}{2}q=0.\]
Thus $n_3=0$, which implies that $d_1=p$, $d_2=q+2$ and $p=q+1$. But
now $d_1=p=q+1<q+2=d_2$, which is also a contradiction to $d_1\ge d_2$.

So by the above discussions, we conclude that $d_1=p+1$, that is,
$n_{p+1}\geq1$. Now, assume that $n_{p+1}\geq 2$, that is, there exist at
least $2$ vertices with  degree $p+1$ in $G'$. Then by
(\ref{sanjiao}), we have
\begin{eqnarray*}
\begin{pmatrix}
            p+1 \\
            3 \\
          \end{pmatrix}
          +\begin{pmatrix}
            q+1 \\
            3 \\
          \end{pmatrix}=\sum_{i=1}^{d_1}\begin{pmatrix}
            i \\
            3 \\
          \end{pmatrix}n_i\geq 2\begin{pmatrix}
            p+1 \\
            3 \\
          \end{pmatrix}+\sum_{i=1}^{p}\begin{pmatrix}
            i \\
            3 \\
          \end{pmatrix}n_i,
\end{eqnarray*}
that is,
\begin{eqnarray*}
\begin{pmatrix}
            q+1 \\
            3 \\
          \end{pmatrix}
          -\begin{pmatrix}
            p+1 \\
            3 \\
          \end{pmatrix}\geq \sum_{i=1}^{p}\begin{pmatrix}
            i \\
            3 \\
          \end{pmatrix}n_i\ge0.
\end{eqnarray*}
This is a contradiction to $q<p$. Hence $n_{p+1}=1$. For
$q\geq4$,   by  (\ref{n_4}) and  (\ref{n_3}), it is easy to see that
$d_2=q+1$. By (\ref{sanjiao}), we have $n_i=0$ for $i=3,4,\ldots,q,
q+2,\ldots,p$.  By (\ref{vertex sum}) and (\ref{edges sum}), we obtain that
$n_1=p+q$ and $n_2=n-2$.
Therefore, $\deg(G')=(p+1, q+1,
\overbrace{2,\ldots,2}^{n-2}, \overbrace{1,\ldots,1}^{p+q})$.

This completes the proof. \qed\end{proof}

\begin{prop}\label{mainprop}
Any double starlike tree $H(p,n,q)$ for $n\geq4$ and $p>q\geq2$ is determined by its Laplacian spectrum.
\end{prop}
\begin{proof}
Let $G$ denote a double starlike tree $H(p,n,q)$ with $n\geq4$ and $p>q\geq2$. Suppose that $G^{\prime}$ is $L$-cospectral to $G$. By Lemma \ref{degreesequenceLem2}, $G^{\prime}$ is a double starlike tree with $\deg(G')=(p+1, q+1,
\overbrace{2,\ldots,2}^{n-2}, \overbrace{1,\ldots,1}^{p+q})$. In the following, we show that $G^{\prime}$ is isomorphic to $G$.

First, we assume that two vertices of degree greater than two are adjacent in $G'$. Suppose that there exist $q-a$ (respectively, $p-b$) pendant vertices adjacent to the vertex of degree $q+1$ (respectively, $p+1$) in $G'$, where $a$ and $b$ are nonnegative integers satisfying that $0\leq a \leq q $ and  $0\leq b\leq p$. Denote by $n_i'$ the number of vertices with degree $i$ in $\mathcal {L}(G^{\prime})$. Then $n_1'=a+b$, $n_{p+q}'=1$, $n_{p+1}'=b$, $n_{p}'=p-b$, $n_{q+1}'=a$, $n_{q}'=q-a$, $n_2'=n-2-a-b$ and $n_{j}'=0$ for $j\notin \{1,2,p,q,p+1,q+1,p+q\}$. On the other hand, it is easy to obtain that $\deg(\mathcal {L}(G))=(p+1,q+1,\overbrace{p,\ldots,p}^p,\overbrace{q,\ldots,q}^q,\overbrace{2,\ldots,2}^{n-3}).$ Note that $\mathcal {L}(G)$ and
$\mathcal {L}(G^{\prime})$ are $A$-cospectral. By (b) and (c) of Lemma \ref{spectrum}, $\mathcal {L}(G)$ and $\mathcal {L}(G^{\prime})$ have the same number of edges and the same number of closed walks of length $4$.
Moreover, they have the same  number of $4$-cycles. Lemma \ref{4
cycles} implies that $\mathcal {L}(G)$ and $\mathcal {L}(G^{\prime})$
have the same number of induced paths of length $2$, that is,
\begin{eqnarray}
  & & \begin{pmatrix}
            p+1 \\
            2 \\
          \end{pmatrix}+\begin{pmatrix}
            q+1 \\
            2 \\
          \end{pmatrix}+p\begin{pmatrix}
             p \\
            2 \\
            \end{pmatrix}+q\begin{pmatrix}
            q \\
            2 \\
           \end{pmatrix}+(n-3)\begin{pmatrix}
            2 \\
            2 \\
          \end{pmatrix}\nonumber\\
& &=\begin{pmatrix}
            p+q \\
            2 \\
          \end{pmatrix}+
       b\begin{pmatrix}
            p+1 \\
            2 \\
          \end{pmatrix}+(p-b)\begin{pmatrix}
            p\\
            2 \\
          \end{pmatrix}+a\begin{pmatrix}
            q+1 \\
            2 \\
          \end{pmatrix}+(q-a)\begin{pmatrix}
            q \\
            2 \\
          \end{pmatrix}+(n-2-a-b)\begin{pmatrix}
            2 \\
            2 \\
          \end{pmatrix}.\nonumber\\
          &&\label{addEqa34}
\end{eqnarray}
Simplifying (\ref{addEqa34}), we obtain that
\[(p-1)(q-1)+b(p-1)+a(q-1)=0.\]
Note that $p>q\geq2$, $0\leq a \leq q $ and $0\leq b\leq p$. We can easily verify that
\[(p-1)(q-1)+b(p-1)+a(q-1)\neq0.\] We then obtain a contradiction.

\begin{figure}[here]
\setlength{\unitlength}{4pt}
\begin{picture}(26,15.5)(-37,-5.8)
\put(0,5){\circle*{1.1}} \put(0,0){\circle*{1.1}}
\put(4.5,-2.3){\footnotesize{1}} \put(9.5,-2.3){\footnotesize{2}}
\put(14,-2.3){\footnotesize{3}}
\put(17.2,0){$\ldots$} 
\put(23.2,-2.3){\footnotesize{$l_1-1$}}\put(31.2,-2.3){\footnotesize{$l_1$}}
\put(-0.25,1.5){$\vdots$} \put(-0.25,-3.5){$\vdots$}
\put(36.75,1.5){$\vdots$} \put(36.75,-3.5){$\vdots$}
\put(0,0){\line(-1,0){5}} \put(-5,0){\circle*{1.1}}
\put(-5.5,1){\footnotesize{1}} \put(-5,0){\line(-1,0){5}}
\put(-10,0){\circle*{1.1}} \put(-10.5,1){\footnotesize{2}}
\put(-14,0){$\ldots$} \put(-15,0){\circle*{1.1}}
\put(-17.5,1){\footnotesize{$l_b''-1$}} \put(-15,0){\line(-1,0){5}}
\put(-20,0){\circle*{1.1}} \put(-20.5,1){\footnotesize{$l_b''$}}
\put(0,-5){\circle*{1.1}} \put(0,-5){\line(-1,0){5}}
\put(-5,-5){\circle*{1.1}} \put(-5.5,-4){\footnotesize{1}}
\put(-5,-5){\line(-1,0){5}} \put(-10,-5){\circle*{1.1}}
\put(-10.5,-4){\footnotesize{2}} \put(-14,-5){$\ldots$}
\put(-15,-5){\circle*{1.1}} \put(-17.5,-4){\footnotesize{$l_1''-1$}}
\put(-15,-5){\line(-1,0){5}} \put(-20,-5){\circle*{1.1}}
\put(-20.5,-4){\footnotesize{$l_1''$}}
\put(-28.5,-1){$p\left\{\begin{array}{c}{~}\\{~}\\{~}\end{array}\right.$}
\put(-24.5,-3.1){$b\left\{\begin{array}{c}{~}\\{~}\end{array}\right.$}
\put(5,0){\circle*{1.1}} \put(10,0){\circle*{1.1}}
\put(15,0){\circle*{1.1}} \put(22,0){\circle*{1.1}}
\put(27,0){\circle*{1.1}} \put(32,0){\circle*{1.1}}
\put(37,0){\circle*{1.1}} \put(37,0){\line(1,0){5}}
\put(42,0){\circle*{1.1}} \put(42,0){\line(1,0){5}}
\put(47,0){\circle*{1.1}} \put(47,1){\footnotesize{2}}
\put(42,1){\footnotesize{1}} \put(48,0){$\ldots$} \put(52,0){\circle*{1.1}}
\put(50,1){\footnotesize{$l_a'-1$}} \put(52,0){\line(1,0){5}}
\put(57,0){\circle*{1.1}} \put(57,1){\footnotesize{$l_a'$}}
\put(37,5){\circle*{1.1}} \put(42,5){\circle*{1.1}}
\put(37,5){\line(1,0){5}} \put(42,5){\line(1,0){5}}
\put(47,5){\circle*{1.1}} \put(47,6){\footnotesize{2}}
\put(42,6){\footnotesize{1}} \put(48,5){$\ldots$} \put(52,5){\circle*{1.1}}
\put(50,6){\footnotesize{$l_1'-1$}} \put(52,5){\line(1,0){5}}
\put(57,5){\circle*{1.1}} \put(57,6){\footnotesize{$l_1'$}}
\put(55,2){$\left.\begin{array}{c}{~}\\{~}\end{array}\right\}a$}
\put(58.5,-1){$\left.\begin{array}{c}{~}\\{~}\\{~}\end{array}\right\}q$}
\put(37,-5){\circle*{1.1}} \put(0,5){\line(1,-1){5}}
\put(0,0){\line(1,0){15}} \put(22,0){\line(1,0){15}}
\put(0,-5){\line(1,1){5}} \put(32,0){\line(1,1){5}}
\put(32,0){\line(1,-1){5}}
\end{picture}
\caption{The graph $G'$}\label{addf4}
\end{figure}
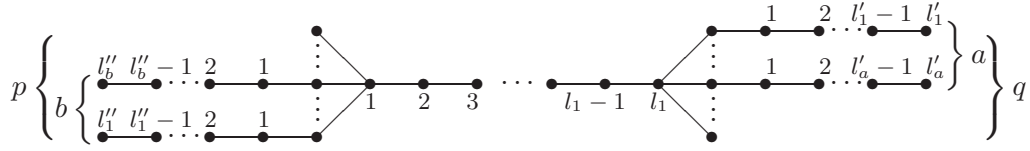

Now assume that two vertices of degree greater than two are not adjacent in $G'$ as shown in Fig. \ref{addf4}.  Again, suppose that there exist $q-a$ (respectively, $p-b$) pendant vertices adjacent to the vertex of degree $q+1$ (respectively, $p+1$) in $G'$, where $a$ and $b$ are nonnegative integers satisfying that $0\leq a \leq q $ and  $0\leq b\leq p$. Then, by counting the number of vertices in $G'$ and $G$, we obtain that
\[l_1+\sum_{i=1}^{a}l_i'+\sum_{j=1}^{b}l_j''+(p+q)=n+p+q,\]
that is,
\begin{equation}\label{sum of cospectral}
l_1+\sum_{i=1}^{a}l_i'+\sum_{j=1}^{b}l_j''=n,
\end{equation}
where $l_1$, $l_i'$ for $1\leq i \leq a$ and $l_j''$ for $1\leq j\leq b$ are positive integers shown in Fig. \ref{addf4}. Again, denote by $n_i'$ the number of vertices with degree $i$ in $\mathcal {L}(G^{\prime})$. Then $n_1'=a+b $,
$n_{p+1}'=b+1$, $n_{p}'=p-b$, $n_{q+1}'=a+1$, $n_{q}'=q-a$, $n_2'=n-3-a-b$ and $n_{j}'=0$ for $j\notin \{1,2,p,q,p+1,q+1\}$. Similar to the discussion above, $\mathcal {L}(G)$ and $\mathcal {L}(G^{\prime})$
have the same number of induced paths of length $2$, that is,
\begin{eqnarray}
  & & \begin{pmatrix}
            p+1 \\
            2 \\
          \end{pmatrix}+\begin{pmatrix}
            q+1 \\
            2 \\
          \end{pmatrix}+p\begin{pmatrix}
             p \\
            2 \\
            \end{pmatrix}+q\begin{pmatrix}
            q \\
            2 \\
           \end{pmatrix}+(n-3)\begin{pmatrix}
            2 \\
            2 \\
          \end{pmatrix}\nonumber\\
& &=(b+1)\begin{pmatrix}
            p+1 \\
            2 \\
          \end{pmatrix}+
       (p-b)\begin{pmatrix}
            p \\
            2 \\
          \end{pmatrix}+(a+1)\begin{pmatrix}
            q+1 \\
            2 \\
          \end{pmatrix}+(q-a)\begin{pmatrix}
            q \\
            2 \\
          \end{pmatrix}+(n-3-a-b)\begin{pmatrix}
            2 \\
            2 \\
          \end{pmatrix}.\nonumber\\
       &&\label{AddEquallss}
\end{eqnarray}
Simplifying (\ref{AddEquallss}), we obtain that
\[a(q-1)+b(p-1)=0,\]
which implies that $a=0$ and $b=0$. Plugging $a=0$ and $b=0$ back into (\ref{sum of
cospectral}), we have $l_1=n$. Therefore, $G'$ is isomorphic to
$G$.

This completes the proof. \qed\end{proof}

\begin{Tproof} \textbf{of Theorem \ref{mainthm}.} The result follows from Propositions \ref{partialprop} and \ref{mainprop} immediately. \qed \end{Tproof}

\begin{rem}
{\em The question of whether any double starlike tree $H(p,n,q)$ is determined by its Laplacian spectrum was first posed in \cite{kn:LiuDouble} after the authors proved that $H(p,n,p)$ is determined by its Laplacian spectrum. Conclusion of \cite{kn:LiuDouble} indicates that the crucial point of this question is to determine $\deg(G')$, where $G'$ is $L$-cospectral with $H(p,n,q)$. We solve this in Lemma \ref{degreesequenceLem2}, which helps us to completely prove that any double starlike tree $H(p,n,q)$ is determined by its Laplacian spectrum. We need to mention that this question was also investigated by other researchers. However, they just solved this question partially. We friendly collect these results as follows:  Any double starlike trees $H(p,n,q)$ satisfying $q=2$ or $p\ge q^2$  or $p-q=1, 2, 7$ is determined by its Laplacian spectrum \cite{kn:Aalipour13,kn:Lu08,kn:Lu09,kn:Lu10,kn:Shen07}.
}
\qed\end{rem}

Since the $L$-spectrum of a graph determines that of its complement \cite{kn:Kelmans65}, Theorem \ref{mainthm} implies the following result readily.

\begin{cor}
The complement of any double starlike tree $H(p,n,q)$  is determined by its Laplacian spectrum.
\end{cor}




\end{document}